\documentclass{amsart}
\pdfoutput=1 
\usepackage[svgnames]{xcolor}
\usepackage{amssymb}
\usepackage{tikz-cd}
\usepackage{microtype}
\usepackage[unicode,
  colorlinks=true,
  linktocpage=true,
  citecolor=ForestGreen,
  linkcolor=MediumOrchid,
  urlcolor=MediumOrchid,
  pdftitle={The Decomposition Space Perspective},
  pdfauthor={Philip Hackney},
  pdfkeywords={decomposition space, 2-Segal space}]{hyperref}
\usepackage[capitalise]{cleveref}
\usepackage{enumitem}
\usepackage{mathtools}
\usepackage{comment}
\usepackage{cite}
\newtheorem{theorem}{Theorem}[section]
\newtheorem{lemma}[theorem]{Lemma}
\newtheorem{proposition}[theorem]{Proposition}
\newtheorem{corollary}[theorem]{Corollary}
\newtheorem*{pastinglaw}{Pasting Law for Pullbacks}
\newtheorem*{retrstab}{Retract Stability}
\theoremstyle{definition}
\newtheorem{definition}[theorem]{Definition}
\newtheorem{example}[theorem]{Example}
\newtheorem{convention}[theorem]{Convention}

\newtheorem{exercise}[theorem]{Exercise}
\theoremstyle{remark}
\newtheorem{remark}[theorem]{Remark}
\newtheorem{diversion}[theorem]{Diversion}

\newcommand{\set}{\mathsf{Set}}
\newcommand{\spaces}{\mathsf{Space}}
\newcommand{\sset}{\mathsf{sSet}}
\newcommand{\sspace}{\mathsf{sSpace}}

\newcommand{\fun}{\mathsf{Fun}}
\newcommand{\catpar}{\mathsf{Par}}
\newcommand{\op}{\textup{op}}
\newcommand{\actrm}{\textup{act}}
\newcommand{\intrm}{\textup{int}}
\newcommand{\delact}{\Delta_\actrm}
\newcommand{\delint}{\Delta_\intrm}
\newcommand{\udec}{\textup{dec}_{\top}}
\newcommand{\ldec}{\textup{dec}_{\perp}}
\DeclareMathOperator{\sd}{sd}
\DeclareMathOperator{\id}{id}
\DeclareMathOperator*{\colim}{colim}

\usepackage[old]{old-arrows}
\usepackage[only,mapsfromchar,mapsfrom]{stmaryrd}
\newcommand{\ract}{\varrightarrow\mapsfromchar}
\newcommand{\rint}{\rightarrowtail}

\tikzset{
  act /.tip = >|
}

\usepackage[
textwidth=3cm,
textsize=small,
colorinlistoftodos,disable]
{todonotes}

\newcommand{\phnote}[1]{\todo[color=purple!40,linecolor=purple!40!black,size=\tiny]{#1}}

\begin{document}

\title{The Decomposition Space Perspective}
\author{Philip Hackney}
\address{Department of Mathematics, University of Louisiana at Lafayette}
\email{philip@phck.net}
\thanks{This work was supported by a grant from the Simons Foundation (\#850849) and by the Louisiana Board of Regents through the Board of Regents Support fund LEQSF(2024-27)-RD-A-31.}

\date{May 9, 2025; visit \url{https://github.com/phck/decomp-space} for updates.}



\begin{abstract}
This paper provides an introduction to decomposition spaces and 2-Segal spaces, unifying the two perspectives.
We begin by defining decomposition spaces using the active-inert factorization system on the simplicial category, and show their equivalence to 2-Segal spaces. 
Key results include the path space criterion, which characterizes decomposition spaces in terms of their upper and lower décalages, and the edgewise subdivision criterion. 
We also introduce free decomposition spaces arising from outer face complexes, providing a rich source of examples. 
Formal prerequisites are minimal -- readers should have a working knowledge of simplicial methods and basic category theory.
\end{abstract}

\maketitle

\section{Introduction}

Decomposition spaces were introduced by Imma G\'alvez-Carrillo, Joachim Kock, and Andrew Tonks in a long preprint in 2014, which now comprises six papers \cite{GKT1,GKT2,GKT3,GKT:DSC,GKT:DSRS,GKT:HLA}.
Their goals were to develop a very general framework for incidence algebras and Möbius inversion that could unify and extend various examples arising in algebraic combinatorics.
Further developments along these lines came later, including consideration of antipodes \cite{CarlierKock:AMDS}, incorporation of Schmitt's hereditary species into the picture \cite{Carlier:HSMDS}, investigation of the universality of Lawvere's Hopf algebra of M\"obius intervals \cite{Forero:GCKTCLDDS}, and the introduction of a new type of species governing other known Hopf algebras \cite{CebrianForero:DHSDS}.

The aim of this paper is to give a gentle introduction to decomposition spaces, similar to Stern's introduction to 2-Segal spaces \cite{Stern:BIRS}.
The theory of decomposition spaces encourages us to reduce as much as possible to a calculus of pullback squares, and we try to follow this principle whenever feasible.
All told, the prerequisites for this work are mild, and similar to \cite{Stern:BIRS}: some elementary category theory, and some comfort with simplicial sets (but see \cref{sec ways to read}).

Decomposition spaces are equivalent to the 2-Segal spaces of Dyckerhoff and Kapranov \cite{DyckerhoffKapranov:HSS}. 
This was first observed for \emph{unital} 2-Segal spaces, as described in \cite[p.~959]{GKT1}:

\begin{quote}
The coincidence of the notions was noticed by Mathieu Anel because two of the basic results are the same: specifically, the characterisation in terms of decalage and Segal spaces \dots{} and the result that the Waldhausen $S_\bullet$-construction of a stable $\infty$-category is a decomposition space \dots{}
\end{quote}
Later on it was realized by Feller, Garner, Kock, Proulx, and Weber \cite{Feller_et_al:E2SSU} that, in fact, the unitality condition is automatic, providing the missing ingredient to see that 2-Segal spaces and decomposition spaces are the same. 
We will give an account of the equivalence in \cref{sec 2-Segl decomp}, following \cite{GKT1,Feller_et_al:E2SSU}.

The next three sections (\S\ref{sec segal spaces}, \S\ref{sec path space}, \S\ref{sec edgewise}) are all about relationships between decomposition spaces and Segal spaces.
The first of these is simple: decomposition spaces are meant to be a generalization of Segal spaces.
We establish that each Segal space is indeed a decomposition space, and that Segal spaces are characterized among the decomposition spaces by checking whether a single square is a pullback.
In \cref{sec path space} we turn to the \emph{path space criterion}, which says that a simplicial space is a decomposition space if and only if both of its d\'ecalages/path spaces are Segal spaces.
Finally, in \cref{sec edgewise} we prove the closely related edgewise subdivision criterion, which says that $X$ is a decomposition space if and only if $\sd X$ is a Segal space.

The last section of the paper describes a source of examples of decomposition spaces, namely the simplicial sets associated to outer face complexes.

There is fundamental material about decomposition spaces that is not covered in depth here, but can be found in the original sources and in other papers in this volume.

\subsection{Ways to read this paper}\label{sec ways to read}
In following the original sources \cite{GKT1,GKT2,GKT3}, we have elected mostly to phrase everything within the $(\infty,1)$-category of simplicial \emph{spaces}.\footnote{For concreteness, `$(\infty,1)$-category' may be interpreted to mean `quasi-category' / `$\infty$-category.' 
That said, as in the original sources (see \cite[1.3]{GKT1}), we work in a model-independent, synthetic fashion, and use nothing specific to quasi-categories.}
But knowing about $(\infty,1)$-categories and their limits and such is not a prerequisite for this paper.
Readers may safely replace ``spaces'' with ``sets'' and replace any infinity-categorical pullbacks/limits with ordinary pullbacks/limits.
In more detail, any of the following are valid ways to read this paper:
\begin{enumerate}[left=0pt]
\item Focus solely on the category $\sset \coloneqq \fun(\Delta^\op, \set)$ of simplicial sets. 
This approach requires basic familiarity with simplicial sets as combinatorial objects (but nothing of their homotopy theory), as well as basic notions from category theory (functors, pullbacks, etc.).
\label{item sset}
\item All statements before \cref{sec free} also hold, using the same arguments, for simplicial objects in a 1-category $C$.\footnote{Typically, one would require that $C$ appearing in \eqref{item gen cat} and \eqref{item gen inf cat} has finite limits.
This assumption can be dropped for what we do in this paper, by instead asking that certain diagrams arising from simplicial objects \emph{are} limit diagrams.}
Particular examples may not make sense in this generality.
\label{item gen cat}
\item Work in the $(\infty,1)$-category $\sspace \coloneqq \fun(\Delta^\op, \spaces)$ of space-valued presheaves on $\Delta$.
Here $\spaces$ denotes the $(\infty,1)$-category of spaces (or $\infty$-groupoids).
See \cite{GKT1} for details.
We emphasize that any pullbacks or limits that appear are necessarily infinity-categorical pullbacks or limits (see \cref{conv pullback}).
\item 
One could also work with simplicial objects in an arbitrary $(\infty,1)$-category $C$.
\label{item gen inf cat}
\end{enumerate}
We will altogether avoid discussing anything from the perspective of Quillen model categories.
A number of sources, especially on the 2-Segal side of things, give model categorical accounts of some of what appears below (see, for instance, \cite{DyckerhoffKapranov:HSS,Feller:Q2SS,BOORS:ESC}).

\subsection*{Acknowledgements}
This paper originated from my talks at the workshop ``Higher Segal Spaces and their Applications to Algebraic K-Theory, Hall Algebras, and Combinatorics'' held at the Banff International Research Station (BIRS) in January 2024.
I thank the workshop organizers (Julie Bergner, Joachim Kock, Maru Sarazola) for cultivating a vibrant and stimulating atmosphere and giving me the opportunity to speak, and BIRS for providing an ideal venue for the event.
I am grateful to the other participants for their contributions and discussions.
Special thanks go to Julie Bergner, Matt Feller, Joachim Kock, Justin Lynd, Felix Naß, and Walker Stern.

\section{The active-inert factorization system on the simplicial category}\label{sec act inert fact}

The \emph{simplicial category} $\Delta$ has objects $[n] = \{0 < 1 < \cdots < n\}$ for $n\geq 0$, and morphisms are the (weakly) order preserving maps (if $i \leq j$, then $\alpha(i) \leq \alpha(j)$).
Let us recall the standard generating maps for this category.
For each $i = 0, 1, \dots, n$, there is a \emph{coface map} $\delta^i \colon [n-1] \to [n]$ which is injective and does not have $i$ in its image, and a \emph{codegeneracy map} $\sigma^i \colon [n+1] \to [n]$ which is surjective with the preimage of $i$ having cardinality two. 
In formulas:
\begin{align*}
  \delta^i(k) &= \begin{cases}
    k & k \leq i-1 \\
    k+1 & k \geq i
  \end{cases}
  &
  \sigma^i(k) &= \begin{cases}
    k & k \leq i \\
    k-1 & k \geq i+1.
  \end{cases}
\end{align*}
A \emph{simplicial object} in category $C$ is a functor $X \colon \Delta^\op \to C$ which is written as $[n] \mapsto X_n$ and $(\alpha \colon [n] \to [m]) \mapsto (\alpha^* \colon X_m \to X_n)$.
As usual, we will write $d_i \colon X_n \to X_{n-1}$ for the image of $\delta^i$, and $s_i \colon X_n \to X_{n+1}$ for the image of $\sigma^i$; these are called \emph{face} and \emph{degeneracy} operators and lead to the usual picture of a simplicial object $X$:
\begin{equation}\label{diag simplicial object} \begin{tikzcd}
X_0 \rar["s_0" description] & X_1 \lar[shift left=2, "d_0"] \lar[shift right=2, "d_1"']  \rar[shift left=1.5] \rar[shift right=1.5] & X_2 \lar[shift left=3,"d_0"] \lar[shift right=3,"d_2"'] \lar
\rar[shift left=3] \rar[shift right=3] \rar &
X_3 
\lar[shift left=1.5] \lar[shift left=4.5] \lar[shift right=1.5] \lar[shift right=4.5] \cdots
\end{tikzcd} \end{equation}
These basic maps define a simplicial object if and only if they satisfy the usual \emph{simplicial identities}
\begin{align*}
d_i d_j &= d_{j-1} d_i \text{ if  $i < j$} \\
s_i s_j &= s_{j+1} s_i \text{ if $i\leq j$} \\
d_i s_j &= \begin{cases}
  s_{j-1} d_i & \text{if $i < j$} \\
  \id & \text{if $i=j$ or $i=j+1$} \\
  s_j d_{i-1} & \text{if $i > j+1$.}
\end{cases}
\end{align*}
Vertical reflection of the picture \eqref{diag simplicial object} yields the following contrivance, which will allow us to infer some statements by duality.

\begin{definition}
If $X$ is a simplicial object, then its \emph{opposite} $X^\op$ has $(X^\op)_n = X_n$ together with the following face and degeneracy operators: \begin{align*} \big( d_k^{\op} \colon X^\op_n \to X^\op_{n-1} \big) &= \big( d_{n-k} \colon X_n \to X_{n-1} \big)
\\
\big( s_k^{\op} \colon X^\op_n \to X^\op_{n+1} \big) &= \big( s_{n-k} \colon X_n \to X_{n+1} \big).
\end{align*}
More formally, there is an endofunctor $\Delta \to \Delta$ which sends $f \colon [n] \to [m]$ to $g \colon [n] \to [m]$ with $g(i) = m-f(n-i)$.
The functor $X^\op \colon \Delta^\op \to \spaces$ is given by precomposing $X$ with this endofunctor.
\end{definition}

We say that a morphism $\alpha \colon [n] \to [m]$ in $\Delta$ is \emph{active} if it is endpoint preserving, that is, if $\alpha(0) = 0$ and $\alpha(n) = m$.
Active maps will denoted by arrows with right arrow to bar 
$\alpha \colon [n] \ract [m]$.
We write $\delact \subset \Delta$ for the subcategory containing the same objects as $\Delta$, but with only the active maps.
A subcategory containing all objects of the larger category is called \emph{wide}.
Notice that, for each $n$, there is a \emph{unique} active map $[1] \ract [n]$.

A morphism $\alpha \colon [n] \to [m]$ in $\Delta$ is \emph{inert} if it is distance preserving, that is, if $\alpha(i+1) = \alpha(i) + 1$ for all $i\in [n]$.
Alternatively, we have $\alpha(\text{-}) = (\text{-}) + c$, where the constant $c$ is $\alpha(0)$.
Inert maps will be written with right arrow tail $\alpha \colon [n] \rint [m]$, and the wide subcategory of inert maps will be denoted $\delint \subset \Delta$.

The only generating morphisms of $\Delta$ which are inert are the \emph{outer coface maps} $\delta^0, \delta^{n} \colon [n-1] \to [n]$.
The codegeneracy maps and the inner coface maps $\delta^i \colon [n-1] \to [n]$ for $0 < i < n$ are active.
The outer face maps in a simplicial set are important enough in what follows that we give them special names: we write
\begin{align*}
  d_\bot \coloneqq d_0 &\colon X_n \to X_{n-1} \\
  d_\top \coloneqq d_n &\colon X_n \to X_{n-1}
\end{align*}
for the bottom and top face maps.
(We also have bottom and top degeneracy maps $s_\bot = s_0$ and $s_\top = s_n \colon X_n \to X_{n+1}$.)

Given an arbitrary map $\alpha \colon [n] \to [m]$ in $\Delta$, there is a unique factorization of $\alpha$ as an active map followed by an inert map.
Namely, setting $p= \alpha(n) - \alpha(0)$, the map $\alpha$ factors as 
\[ \begin{tikzcd}[sep=huge]
{[n]} \ar[rr, bend left=20, "\alpha"] \rar[-act, "\alpha(\text{-}) - \alpha(0)"'] & {[p]} \rar[tail, "(\text{-}) + \alpha(0)"'] & {[m].}
\end{tikzcd} \]
The first map applies $\alpha$ and then subtracts the constant $\alpha(0)$, while second map adds the constant $\alpha(0)$.
It is immediate that the first map is active and the second map is inert, and that the composite of the two is equal to $\alpha$.
A bit more thought will show that this is the \emph{only} active-inert factorization of $\alpha$, and we arrive at the following.
\begin{lemma}
The pair of subcategories $(\delact, \delint)$ constitutes a factorization system on $\Delta$. \qed
\end{lemma}
An \emph{orthogonal factorization system}\footnote{The name `orthogonal' comes from the following lifting property: for every commutative square  $\begin{tikzcd}[cramped, sep=small, ampersand replacement=\&] \cdot \dar[swap]{\ell} \rar \& \cdot \dar{r} \\ \cdot \rar \urar[dotted] \& \cdot  \end{tikzcd}$  with $\ell \in L$ and $r\in R$, there exists a unique dotted map making both triangles commute.} on a category $C$ consists of a pair of wide subcategories $(L,R)$ such that every map in $C$ factors as $r\circ \ell$ for some $r\in R$ and $\ell \in L$, this factorization is unique up to unique isomorphism, and $L\cap R$ contains all isomorphisms.
A \emph{strict factorization system} is similar, except that we only require one condition of the wide subcategories: every map factors \emph{uniquely} as $r\circ \ell$.
In general, the two notions are not directly comparable (though there is an enveloping construction for strict factorization systems \cite[2.1]{Grandis:WSEMCC}, \cite[3.1]{RosebrughWood:DLF}), but in our case there is no ambiguity as the only isomorphisms in $\Delta$ are identities.
Generalizations of this factorization system to other shape categories are an essential component of Segalic approaches to higher structures \cite{Berger:MCO,ChuHaugseng:HCASC,Hackney:SCGO} (see \cref{def Segal} below).

In the next section, we will make use of a special property of $\Delta$, which is that active-inert pushouts exist.
\begin{lemma}\label{lem ai pushout}
Suppose $\alpha \colon [n] \ract [m]$ is active and $\iota \colon [n] \rint [k]$ is inert.
Then there is a pushout square
\begin{equation}\label{diag ai pushout} \begin{tikzcd}
{[n]} \rar[-act, "\alpha"] \dar[tail, "\iota"'] \ar[dr, phantom, "\ulcorner" very near end] & {[m]} \dar[tail, "\theta"] \\
{[k]} \rar[-act, "\phi"'] & {[p]}
\end{tikzcd} \end{equation}
in $\Delta$, with $\theta$ inert and $\phi$ active. \qed
\end{lemma}
For brevity we leave the proof as an exercise for the reader; 
see \cite[2.6]{GKT1} for this fact, or \cite[\S3]{ConstantinFritzPerroneShapiro:WCPSS} for a discussion of \emph{all} pushout squares which exist in $\Delta$.
Concrete formulas for $\theta$ and $\phi$ allow one to see that \eqref{diag ai pushout} is also a pullback square.

\section{Decomposition spaces and 2-Segal spaces}\label{sec decomp vs 2-segal}

With active-inert pushouts in hand, we are ready to define decomposition spaces \cite[\S3]{GKT1}.
While categories encode associative and unital \emph{composition} of morphisms, decomposition spaces shift attention to possible \emph{decompositions} of morphisms, requiring these to be associative and unital in an appropriate sense.
This idea can be used to encode a number of coalgebras which come from decomposition of combinatorial structures into smaller pieces (see \cref{ex forests}) via a general incidence coalgebra construction \cite{CooperYoung:BIRS,GKT1,GKT:DSC}.

\begin{definition}[Decomposition space]\label{def decomp space}
A simplicial space $X$ is a \emph{decomposition space} if it sends every active-inert pushout square to a pullback square.
\end{definition}

In more detail, given a pushout square on the left below (as in \eqref{diag ai pushout} of \cref{lem ai pushout}), the associated square on the right is a pullback square of spaces.
\[
\begin{tikzcd}
{[n]} \rar[-act, "\alpha"] \dar[tail, "\iota"'] \ar[dr, phantom, "\ulcorner" very near end] & {[m]} \dar[tail, "\theta"] 
\ar[drr,phantom,"\rightsquigarrow"]
& & 
X_p \rar{\phi^*} \dar[swap]{\theta^*} 
\ar[dr, phantom, "\lrcorner" very near start]
& X_k \dar{\iota^*} 
\\
{[k]} \rar[-act, "\phi"'] & {[p]} & & 
X_m \rar[swap]{\alpha^*} & X_n
\end{tikzcd}
\]

\begin{convention}\label{conv pullback}
In this definition, and in what follows, `pullback' means `pullback in the $(\infty,1)$-category of spaces' or `homotopy pullback.'
For a reader working with simplicial sets rather than simplicial spaces, these will just be ordinary pullbacks.
\end{convention}

A simplicial set which is a decomposition space will be called a \emph{discrete decomposition space}.

\begin{diversion}[Variations]
What happens if one replaces the active-inert factorization system on $\Delta$ in \cref{def decomp space} with a factorization system on some other indexing category?
This question is broad and vague, but special cases have been considered in \cite[Remark 3.15]{Berger:MCO}, \cite[p.\ 635]{Burkin:TACOSC}, \cite{Contreras_et_al:FCPBS}, and \cite{Godicke:IC2SS}.\footnote{It does not seem to be the case that cyclic 2-Segal objects (see \cite{DyckerhoffKapranov:TSTC,BergnerStern:CSS}) are of this form, owing to the lack of a compatible orthogonal factorization systems on Connes' cyclic category $\Lambda$ \cite{Connes:CCFE}.
Indeed, if the right class of a factorization system on $\Lambda$ contains $\delint$, then it contains the positive part of the generalized Reedy factorization system on $\Lambda$ (see \cite{BergerMoerdijk:OENRC}), while if the left class of a factorization system contains $\delact$, then it contains all of $\Lambda$.}
\end{diversion}

If $C$ is a category, then the nerve of $C$ is a (discrete) decomposition space.
We return to this example in \cref{sec segal spaces}.
A rich source of combinatorial examples are the free decomposition spaces, which we discuss in \cref{sec free}.
We will give additional examples in a moment, but first we give an equivalent characterization (to be discussed in the next section).
The following appears in \cite{Feller_et_al:E2SSU}.

\begin{definition}\label{def 2-segal}
A simplicial space $X$ is said to be \emph{upper 2-Segal} (resp.\ \emph{lower 2-Segal}) if the square below left (resp.\ below right) is a pullback for all $0 < i < n$.
\begin{equation} \label{eq upper and lower 2-Segal} \begin{tikzcd}
X_{n+1} \rar{d_{i+1}} \dar[swap]{d_\bot} \ar[dr, phantom, "\lrcorner" very near start] & X_n \dar{d_\bot} 
& & 
X_{n+1} \rar{d_i} \dar[swap]{d_\top} \ar[dr, phantom, "\lrcorner" very near start] & X_n \dar{d_\top}
\\
X_n \rar[swap]{d_i} & X_{n-1}
& & 
X_n \rar[swap]{d_i} & X_{n-1}
\end{tikzcd} 
\tag{$\heartsuit$}
\end{equation}
A \emph{2-Segal space} is a simplicial space which is both upper and lower 2-Segal.
\end{definition}

As a simplicial space $X$ is lower 2-Segal if and only if $X^\op$ is upper 2-Segal, many results below address only the upper 2-Segal situation, with the understanding that there is always a dual version.
This definition of 2-Segal space isn't exactly what appeared in \cite{Stern:BIRS}, but will turn out to be equivalent.
We leave this to the reader once we've introduced the relevant tools --
see \cref{exc 2-Segal defs}.
Also note that upper and lower 2-Segal spaces are independently of interest: Garner--Kock--Weber \cite{GarnerKockWeber:OCD} showed that unary operadic categories are the same thing as upper 2-Segal sets (see also \cite{Hackney:OC2SS}), while Carawan \cite{Carawan:2SMACWC} gives examples of lower 2-Segal objects coming from the $S_\bullet$-construction applied to a cofibration category.

\begin{example}[Simplicial groupoid of rooted forests]\label{ex forests}
A fundamental example of \cite{GKT1} concerns the simplicial groupoid $H$ which underlies the Butcher--Connes--Kreimer Hopf algebra \cite{ConnesKreimer:HARNG}.
For $n \geq 0$, $H_{n+1}$ is the groupoid of forests together with $n$ (compatible) admissible cuts.
For example, the following forest is an object of $H_1$,
\begin{center}
\includegraphics[scale=0.2]{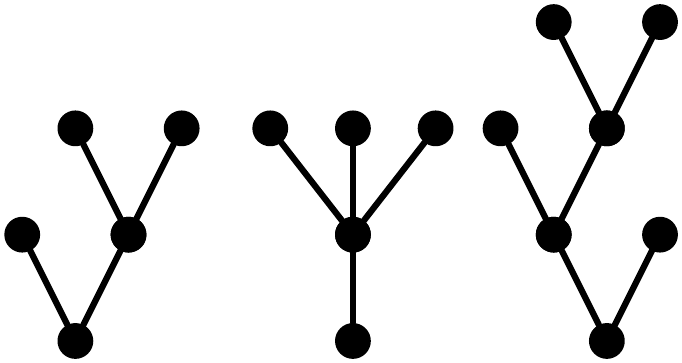}
\end{center}
which is in turn the image under $d_1$ of the following forest with a cut in $H_2$.
\begin{center}
\includegraphics[scale=0.2]{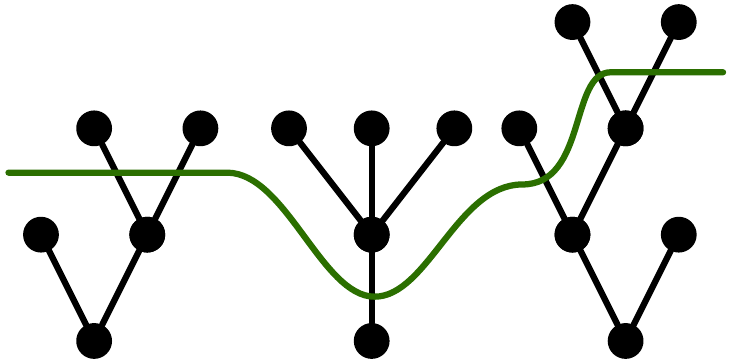}
\end{center}
The outer faces clip above and below the cut, producing 
\begin{center}
\includegraphics[scale=0.2]{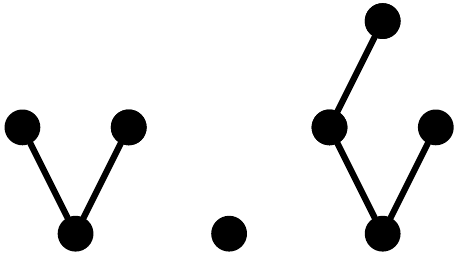} \qquad and \qquad \includegraphics[scale=0.2]{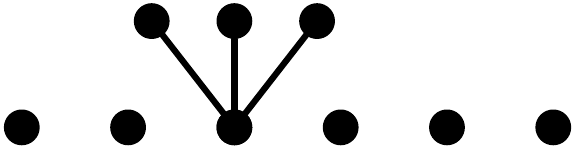} 
\end{center}
This is a 2-Segal space (but not a Segal space in the sense of \cref{def Segal} below, as there are many 2-simplices having the above forests as their outer faces).
As an illustration, the tree with two cuts in the upper left of the following square (which is an element of $H_3$) is uniquely determined by the remainder of the diagram, giving an indication of why the right square of \eqref{eq upper and lower 2-Segal} is a pullback in the case $n=2$.
\begin{center}
\begin{tikzpicture}[commutative diagrams/every diagram]
\node[inner sep = 4pt, outer sep=2pt] (nw) at (0,0)
  {\includegraphics[scale=0.2]{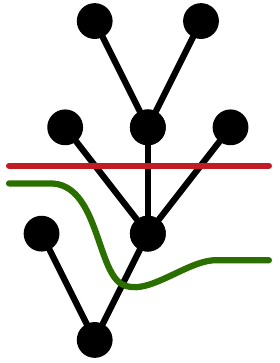}};
\node[inner sep = 4pt, outer sep=2pt] (ne) at (3,0)
  {\includegraphics[scale=0.2]{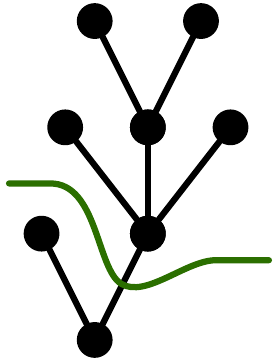}};
\node[inner sep = 4pt, outer sep=2pt] (sw) at (0,-2.5)
  {\includegraphics[scale=0.2]{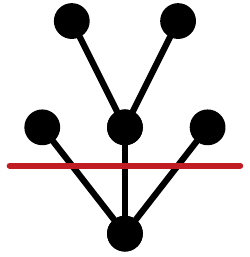}};
\node[inner sep = 4pt, outer sep=2pt] (se) at (3,-2.5)
  {\includegraphics[scale=0.2]{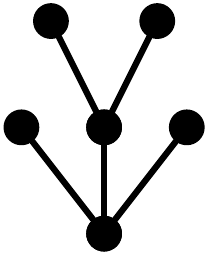}};

\path[commutative diagrams/.cd, every arrow, every label]
  (nw) edge node {$d_1$} (ne)
  (nw) edge node[swap] {$d_3$} (sw)
  (ne) edge node {$d_2$} (se)
  (sw) edge node[swap] {$d_1$} (se);
\end{tikzpicture}
\end{center}
\end{example}

\begin{example}[Partial monoids]\label{ex partial monoid} In \cite[\S2]{Segal:CSILS}, a \emph{partial monoid} is defined to be a set $M$ together with an identity element $1\in M$ and a partially defined multiplication function $M \times M \nrightarrow M$.
These are unital in the sense that $1\cdot x = x = x \cdot 1$ are defined and equal for all $x\in M$, and are associative in the sense that if one of $x \cdot (y \cdot z)$ or $(x \cdot y) \cdot z$ is defined, then so is the other, and they are equal.
This forms a simplicial set where the $n$-simplices are the length $n$ words $(x_1, \dots, x_n)$ such that the multiplication of these $n$ elements is defined for some parenthesization (and hence for all possible parenthesizations). 
It was shown in \cite[Example 2.1]{BOORS:2SSWC} that this simplicial set is always 2-Segal.
\end{example}

\begin{exercise}
Show that every partial monoid is 2-Segal using only \cref{def 2-segal}.
\end{exercise}

This can be extended to a kind of partial category. 
Let $\catpar$ denote the category whose objects are sets and whose morphisms are partial functions, considered as a monoidal category using the cartesian product of sets.
\phnote{Consider removing partial categories.}

\begin{exercise}
Let $C$ be a small category enriched in $\catpar$.
Show that if the set $C(x,x)$ is inhabited for every object $x$ of $C$, then $C$ determines a simplicial set which is 2-Segal.
\end{exercise}

\begin{remark}
Each partial monoid may be regarded as a monoid with a zero element, by adjoining a new absorbing element $0$ to $M$ and declaring that all undefined multiplications take $0$ as their value.
Likewise, for any monoid with a zero element different from $1$, we obtain a partial monoid by discarding $0$.
In a similar manner, the $\catpar$-enriched categories from the preceding exercise correspond to categories with (nonidentity) zero morphisms \cite[\S1.7]{Pareigis:CF}.
\end{remark}

The examples of partial monoids and $\catpar$-enriched categories suggest a useful perspective: simplicial sets which are 2-Segal (or equivalently, discrete decomposition spaces) can be understood as encoding multivalued ``associative'' composition laws.
See \cite[\S3.3]{DyckerhoffKapranov:HSS} and \cite{Stern:BIRS} for precise statements.
The preceding examples cover the case of composition having at most one value, and the single-valued composition laws correspond to the Segal objects (\cref{sec segal spaces}).

We mention, in passing, a class of maps which plays an important role in the theory of decomposition spaces.
These maps will not make a significant appearance in this paper, but will reappear elsewhere in this volume \cite{CooperYoung:BIRS,GKT:DSC}.
In particular, \cite{CooperYoung:BIRS} explains why a culf map between decomposition spaces induces a coalgebra homomorphism between the associated incidence coalgebras (originally from \cite[Lemma 8.2]{GKT1}).

\begin{definition}[Culf maps]\label{def culf}
A map $f \colon X \to Y$ between simplicial spaces is \emph{culf} if it is cartesian on active maps.
Concretely, this means that the following squares are pullbacks, for $0 < i < n$ and $0\leq j \leq n$.
\[ \begin{tikzcd}
X_n \rar{d_i} \dar[swap]{f} \ar[dr, phantom, "\lrcorner" very near start]  & X_{n-1} \dar{f} 
& &
X_n \rar{s_j} \dar[swap]{f} \ar[dr, phantom, "\lrcorner" very near start]  & X_{n+1} \dar{f} 
\\
Y_n  \rar[swap]{d_i} & Y_{n-1} 
& &
Y_n  \rar[swap]{s_j} & Y_{n+1}
\end{tikzcd}
\]
\end{definition}

\section{2-Segal spaces are decomposition spaces}\label{sec 2-Segl decomp}

In this section we discuss the following fundamental result; as mentioned in the introduction, it was known early on that decomposition spaces were the same thing as unital 2-Segal spaces, and later on it was discovered that the unitality condition is automatic \cite{Feller_et_al:E2SSU}.

\begin{theorem}\label{thm unitality} 
A simplicial space $X$ is a decomposition space if and only if it is both upper and lower 2-Segal.
\end{theorem}

It is immediate that every decomposition space is 2-Segal, since the squares in \eqref{eq upper and lower 2-Segal} of \cref{def 2-segal} are active-inert pullback squares.
The converse is more interesting.
If $X$ is a decomposition space, then the following basic squares (for all $0 \leq i \leq n$) come from active-inert pushout squares in $\Delta$, so must be pullbacks.
\[ \begin{tikzcd}
X_{n+1} \rar{s_{i+1}} \dar[swap]{d_\bot} \ar[dr, phantom, "\lrcorner" very near start]  & X_{n+2} \dar{d_\bot} 
& &
X_{n+1} \rar{s_i} \dar[swap]{d_\top} \ar[dr, phantom, "\lrcorner" very near start]  & X_{n+2} \dar{d_\top} 
\\
X_n  \rar[swap]{s_i} & X_{n+1} 
& &
X_n  \rar[swap]{s_i} & X_{n+1}
\end{tikzcd}
\]
However, the 2-Segal condition makes no mention of any squares involving degeneracies.
The main effort of the proof of \cref{thm unitality} is addressing this.
Most proofs in the rest of the paper rely on two basic tools.
Here is the first:

\begin{pastinglaw}
Suppose we have a diagram
\[ \begin{tikzcd}
\bullet \rar \dar & \bullet \rar \dar \ar[dr, phantom, "\lrcorner" very near start] & \bullet \dar \\
\bullet \rar & \bullet \rar & \bullet 
\end{tikzcd} \]
with right square a pullback.
Then the left square is a pullback if and only if the outer (composite) square is a pullback.
\end{pastinglaw}

For a 1-category $C$ (such as that of sets), this is traditionally an exercise, but see \cite[Proposition 2.5.9]{Borceux:HCA1} for a proof.
In an $(\infty,1)$-category $C$ (such as that of spaces), this is sometimes called the \emph{prism lemma}, since the diagram above should be specified by a functor $\Delta^2 \times \Delta^1 \to C$ (see \cite[Lemma 4.4.2.1]{Lurie:HTT} or \cite[\href{https://kerodon.net/tag/03FZ}{Tag 03FZ}]{kerodon}).

As an illustration of the utility of the pasting law, we establish the following characterization of the (upper) 2-Segal condition from \cite[Lemma 3.6]{GKT1}.

\begin{proposition}\label{prop fewer squares}
The following are equivalent for a simplicial space $X$.
\begin{enumerate}
\item $X$ is upper 2-Segal.\label{item fewer upper}
\item For each $n\geq 2$, the left square of \eqref{eq upper and lower 2-Segal} is a pullback for some $0 < i < n$.\label{item fewer some}
\end{enumerate}
\end{proposition}
\begin{proof}
It is immediate that \eqref{item fewer upper} implies \eqref{item fewer some}.
To prove the converse, it is helpful to add a third equivalent condition:
\begin{enumerate}[start=3]
\item For each $n\geq 2$, the square below is a pullback.\label{item fewer composite}
\[ \begin{tikzcd}
X_{n+1} \rar{d_2^{n-1}} \dar[swap]{d_\bot} & X_2 \dar{d_\bot} \\
X_n \rar[swap]{d_1^{n-1}} & X_1
\end{tikzcd} \]
\end{enumerate}
The following diagram is used for the remaining implications.
\[ \begin{tikzcd}
X_{n+2} \rar{d_{i+1}} \dar[swap]{d_\bot} &
X_{n+1} \rar{d_2^{n-1}} \dar[swap]{d_\bot} \ar[dr, phantom, "\lrcorner" very near start] & X_2 \dar{d_\bot}
\\
X_{n+1} \rar[swap]{d_i} & 
X_n \rar[swap]{d_1^{n-1}} & X_1
\end{tikzcd} \]
For \eqref{item fewer some} implies \eqref{item fewer composite}, we assume inductively that the right square above is a pullback, and that for some $0 < i < n+1$ the left square is a pullback. 
This implies that the outer rectangle is also a pullback.
But $d_2^{n-1}d_{i+1} = d_2^n$ and $d_1^{n-1}d_i = d_1^n$, so this outer rectangle is the case of \eqref{item fewer composite} in the next level.
For \eqref{item fewer composite} implies \eqref{item fewer upper}, let $0 < i < n+1$ be arbitrary. 
Then the outer rectangle and the right square in the diagram are pullbacks, hence so is the left square.
Thus $X$ is upper 2-Segal.
\end{proof}


\begin{lemma}\label{lem higher degen}
Suppose $X$ is an upper 2-Segal space.
If $n > 0$ and $0\leq i \leq n$, then the square
\[ \begin{tikzcd}
X_{n+1} \rar{s_{i+1}} \dar[swap]{d_\bot} \ar[dr, phantom, "\lrcorner" very near start]  & X_{n+2} \dar{d_\bot} \\
X_n  \rar[swap]{s_i} & X_{n+1}
\end{tikzcd} 
\]
is a pullback.
\end{lemma}
\begin{proof}
This is an application of the pasting law.
We can choose $j \in \{i,i+1 \}$ with $0 < j < n+1$, and then we have the following diagram.
\[
\begin{tikzcd}
X_{n+1} \ar[rr,bend left, "\id"'] \rar{s_{i+1}} \dar[swap]{d_\bot} & X_{n+2} \dar[swap]{d_\bot} \ar[dr, phantom, "\lrcorner" very near start] \rar{d_{j+1}} & X_{n+1} \dar{d_\bot} \\
X_n  \rar[swap]{s_i} \ar[rr, bend right, "\id"] & X_{n+1} \rar[swap]{d_j} & X_n
\end{tikzcd}
\]
The right square is a pullback since $X$ is upper 2-Segal, and the outer square is a pullback.
Hence the left square is a pullback.
\end{proof}

Unfortunately, the proof of the lemma does not cover the $n=0$ case.
For that, we need our second basic tool.

\begin{retrstab}
Retracts of pullbacks are pullbacks.
\end{retrstab}
Being a pullback is a property of an object of $\fun(\Delta^1 \times \Delta^1, C)$, the category of commutative squares in $C$, and this property is stable under retracts.\phnote{May be confusing if not working in $\infty$-categories.}
See \cite[Appendix A]{Feller_et_al:E2SSU} for homotopy pullbacks in model categories; for quasi-categories use the duals of \cite[Lemma 5.1.6.3]{Lurie:HTT} or \cite[\href{https://kerodon.net/tag/05E6}{Tag 05E6}]{kerodon}.
The following is proved in \cite{Feller_et_al:E2SSU}.

\begin{lemma}\label{lem bot degen}
If $X$ is an upper 2-Segal space, then the square
\[ \begin{tikzcd}
X_1 \rar{s_1} \dar[swap]{d_0} \ar[dr, phantom, "\lrcorner" very near start]  & X_2 \dar{d_0} \\
X_0  \rar[swap]{s_0} & X_1
\end{tikzcd} 
\]
is a pullback.
\end{lemma}
\begin{proof}
The indicated square is a retract of 
\[
\begin{tikzcd}
X_2 \rar{s_2} \dar[swap]{d_\bot} \ar[dr, phantom, "\lrcorner" very near start]  & X_3 \dar{d_\bot} \\
X_1  \rar[swap]{s_1} & X_2
\end{tikzcd} 
\]
which is known to be a pullback by \cref{lem higher degen}.
This is witnessed by the maps of squares
\[
\setlength\arraycolsep{1pt}
S = \,
\begin{matrix}
s_1 & \vline & s_1 \\
\hline
s_0 & \vline & s_0
\end{matrix}
\qquad
\text{and}
\qquad
D = 
\, \begin{matrix}
d_1 & \vline & d_1 \\
\hline
d_0 & \vline & d_0
\end{matrix}
\]
with $DS = \id$.
More explicitly, maps between squares are cubes, and the diagram
  \[
  \begin{tikzcd}[sep={30pt,between origins}]
      {X_1} \ar[rrr, "s_1"] \ar[dd, "d_0"] \ar[dr, "s_1"] & & & 
      {X_2} \ar[rrr, "d_1"] \ar[dd, "d_0"' very near start] \ar[dr, "s_2"] & & & 
      {X_1} \ar[dd, "d_0"' very near start] \ar[dr, "s_1"] & \\ &
      {X_2} \ar[rrr, pos=0.4, "s_1", crossing over]   & & & 
      {X_3} \ar[rrr, pos=0.4, "d_1", crossing over]   & & & 
      {X_2} \ar[dd,"d_0"] \\ 
      {X_0} \ar[rrr, pos=0.64, "s_0"] \ar[dr, "s_0"'] & & & 
      {X_1} \ar[rrr, pos=0.64, "d_0"] \ar[dr, "s_1"'] & & & 
      {X_0} \ar[dr, "s_0"] & \\ &
      {X_1} \ar[rrr, "s_0"] \ar[from=uu, crossing over, "d_0" very near start] & & &  
      {X_2} \ar[rrr, "d_0"] \ar[from=uu, crossing over, "d_0" very near start] & & &  
      {X_1}
  \end{tikzcd}
  \]
is commutative.
\end{proof}

Dually, if $X$ is a lower 2-Segal space then for $n\geq 0$ and $0\leq i \leq n$, the square 
\[
\begin{tikzcd}
X_{n+1} \rar{s_i} \dar[swap]{d_\top} \ar[dr, phantom, "\lrcorner" very near start]  & X_{n+2} \dar{d_\top} \\
X_n  \rar[swap]{s_i} & X_{n+1}
\end{tikzcd}
\]
is a pullback: apply \cref{lem higher degen} and \cref{lem bot degen} to $X^\op$.
To prove \cref{thm unitality}, one can now follow \cite[3.4]{GKT1}, or reason about the situation directly as follows:

\begin{proof}[Proof sketch for \cref{thm unitality}]
First consider active-inert pushout squares
\[ \begin{tikzcd}
{[n]} \rar[-act, "\alpha"] \dar[tail, "\delta^k"'] \ar[dr, phantom, "\ulcorner" very near end] & {[m]} \dar[tail, "\theta"] \\
{[n+1]} \rar[-act, "\phi"'] & {[m+1]}
\end{tikzcd} \]
where $k$ is either $0$ or $n+1$ (implying $\theta = \delta^\ell$ for $\ell$ either $0$ or $m+1$).
By factoring $\alpha$ into degeneracies and inner faces, we see that 
\begin{equation}\label{diag codim one inert}
\begin{tikzcd}
X_{m+1} \rar{\phi^*} \dar[swap]{d_\ell} 
\ar[dr, phantom, "\lrcorner" very near start]
& X_{n+1} \dar{d_k} 
\\
X_m \rar[swap]{\alpha^*} & X_n
\end{tikzcd}
\end{equation}
is a pullback using the pasting law and our known pullback squares.
Then, for an arbitrary active-inert pushout square as in \eqref{diag ai pushout} of \cref{lem ai pushout}, use induction on $k-n$ along with the squares \eqref{diag codim one inert} and the pasting law. 
\end{proof}

\begin{exercise}\label{exc 2-Segal defs}
Recall that one characterization of 2-Segal spaces given in \cite{Stern:BIRS} is that $X$ sends the pushout diagram of finite ordered sets
\begin{equation}
\label{diag ij}
\begin{tikzcd}
\{ i,j\} \rar[-act] \dar[tail] & \{i,i+1, \dots, j \} \dar[tail] \\
\{0,1,\dots, i,j,j+1, \dots, n\} \rar[-act] & {[n]}
\end{tikzcd} \end{equation}
to a pullback for each $0 \leq i < j \leq n$.
Prove this is equivalent to the definition of 2-Segal space from \cref{def 2-segal}.
\end{exercise}

The squares \eqref{diag ij} correspond to those squares appearing in \eqref{diag ai pushout} of \cref{lem ai pushout} where $n=1$ and $\alpha$ is injective.
As in \cite[Proposition 2.3.2]{DyckerhoffKapranov:HSS}, one only needs to require these squares be sent to pullbacks when $i=0$ or $j=n$.

\section{Segal spaces are decomposition spaces}\label{sec segal spaces}

In \cite{Stern:BIRS}, we saw that the nerve of a category is a 2-Segal set (under the characterization from \cref{exc 2-Segal defs}).
This is true more generally: Segal spaces \cite{Rezk:MHTHT,Segal:CCT} are always decomposition spaces.
In this brief section we use the pasting law to establish this fact, and to characterize the Segal spaces among decomposition spaces.

\begin{definition}\label{def Segal}
A simplicial space $X$ is a \emph{Segal space} if the following equivalent conditions hold:
\begin{enumerate}
\item \label{item pullback Segal} For each $n\geq 1$, the following square is a pullback. 
\[ \begin{tikzcd}
X_{n+1} \ar[dr, phantom, "\lrcorner" very near start] \rar{d_\bot} \dar[swap]{d_\top} & X_n \dar{d_\top} \\
X_n \rar[swap]{d_\bot} & X_{n-1}
\end{tikzcd} \]
\item For each $n\geq 2$, the map
\[
\begin{tikzcd}[column sep=2.2cm]
X_n \rar["(d_\top^{n-i} d_\bot^{i-1})_{i=1}^n" swap, "\simeq"] & X_1 \times_{X_0} X_1 \times_{X_0} \dots  \times_{X_0} X_1
\end{tikzcd} 
\]
is an equivalence. \label{item usual Segal}
\end{enumerate}
\end{definition}

\begin{exercise}
Prove that \eqref{item pullback Segal} and \eqref{item usual Segal} in \cref{def Segal} are equivalent. (For a hint or solution, see \cite[Lemma 2.10]{GKT1}.)
\end{exercise}

\begin{proposition}\label{segal implies decomposition}
Every Segal space is a decomposition space.
\end{proposition}
\begin{proof}
We will show that every Segal space is upper 2-Segal.
Since the opposite of a Segal space is a Segal space, and the opposite of an upper 2-Segal space is a lower 2-Segal space, this will suffice.
By \cref{prop fewer squares}, we need only show that the left square in the following diagram is a pullback.
The right square is a pullback since $X$ is Segal.
\[ \begin{tikzcd}
X_{n+1} \rar{d_n} \dar[swap]{d_\bot} &
X_n \rar{d_\top} \dar[swap]{d_\bot} \ar[dr, phantom, "\lrcorner" very near start] & X_{n-1} \dar{d_\bot}
\\
X_n \rar[swap]{d_{n-1}} & 
X_{n-1} \rar[swap]{d_\top} & X_{n-2}
\end{tikzcd} \]
The outer rectangle is equal to the outer rectangle of the following diagram.
\[ \begin{tikzcd}
X_{n+1} \rar{d_\top} \dar[swap]{d_\bot} \ar[dr, phantom, "\lrcorner" very near start] &
X_n \rar{d_\top} \dar[swap]{d_\bot} \ar[dr, phantom, "\lrcorner" very near start] & X_{n-1} \dar{d_\bot}
\\
X_n \rar[swap]{d_\top} & 
X_{n-1} \rar[swap]{d_\top} & X_{n-2}
\end{tikzcd} \]
Since both squares are pullbacks, so is the outer rectangle.
This implies that the left square in first diagram is a pullback, as desired.
\end{proof}

\begin{proposition}
A simplicial space $X$ is a Segal space if and only if it is a decomposition space and the following square is a pullback.
\[ \begin{tikzcd}
X_2 \rar{d_0} \dar[swap]{d_2} & X_1 \dar{d_1} \\
X_1 \rar[swap]{d_0} & X_0
\end{tikzcd} \]
\end{proposition}
\begin{proof}
The forward direction is immediate from \cref{segal implies decomposition}.
For the reverse direction, one shows via induction that the squares in \cref{def Segal} \eqref{item pullback Segal} are pullbacks; the case $n=1$ is an assumption.
For the inductive step, use the two rectangles from the proof of \cref{segal implies decomposition} in a different way, assuming that both squares in the first rectangle and the right square in the second rectangle are pullbacks, in order to deduce that the left square in the second rectangle is a pullback.
\end{proof}

\section{The path space criterion}\label{sec path space}

We now turn to the path space criterion, which is originally due to Dyckerhoff--Kapranov \cite[Theorem 6.3.2]{DyckerhoffKapranov:HSS}.
Our presentation separates the lower and upper 2-Segal conditions, following Poguntke \cite{Poguntke:HSSAKT}, who also gave higher-dimensional versions.
A version for decomposition spaces appeared as \cite[Theorem 4.10]{GKT1}.

\begin{theorem}\label{thm upper path}
A simplicial space $X$ is upper 2-Segal if and only if its upper d\'ecalage $\udec X$ is Segal.
\end{theorem}

The upper d\'ecalage $\udec X$ is given by precomposition of $X \colon \Delta^\op \to \spaces$ with the endofunctor $\Delta \to \Delta$ that sends $[n]$ to $[n]\star [0] \cong [n+1]$.\footnote{
Here $\star$ denotes the \emph{join} of categories \cite[\href{https://kerodon.net/tag/0160}{Tag 0160}]{kerodon}, where the set of objects of $C \star D$ is the disjoint union of the sets of objects of $C$ and $D$, the categories $C$ and $D$ are full subcategories of $C\star D$, and, for $c\in C$ and $d\in D$, there are no morphisms $d\to c$ and exactly one morphism $c\to d$.}
This endofunctor takes $f\colon [n] \to [m]$ to $g\colon [n+1] \to [m+1]$ satisfying $g(n+1) = m+1$ and $g(k) = f(k)$ for $0\leq k \leq n$.

More explicitly, we have $(\udec X)_n = X_{n+1}$ and the face and degeneracy operators are given for $0 \leq k \leq n$ by
\begin{align*} 
\big( d_k \colon (\udec X)_n \to (\udec X)_{n-1} \big) &= \big( d_k \colon X_{n+1} \to X_n \big)
\\
\big( s_k \colon (\udec X)_n \to (\udec X)_{n+1} \big) &= \big( s_k \colon X_{n+1} \to X_{n+2} \big).
\end{align*}
Notably, we have forgotten about the top face $d_\top = d_{n+1} \colon X_{n+1} \to X_n$ and top degeneracy $s_\top = s_n \colon X_n \to X_{n+1}$.
The upper d\'ecalage is also called the \emph{final path space} $P^\triangleright X$.

\begin{proof}[Proof of \cref{thm upper path}]
For $n\geq 2$, the following two squares are the same.
\[ \begin{tikzcd}
(\udec X)_n \rar{d_\top} \dar[swap]{d_\bot}  & (\udec X)_{n-1} \dar{d_\bot} 
&  
X_{n+1} \rar{d_n} \dar[swap]{d_\bot}  & X_n \dar{d_\bot}
\\
(\udec X)_{n-1} \rar[swap]{d_\top} & (\udec X)_{n-2}
&  
X_n \rar[swap]{d_{n-1}} & X_{n-1}
\end{tikzcd} \]
\cref{prop fewer squares} says that the right square is a pullback for all $n\geq 2$ if and only if $X$ is upper 2-Segal.
\end{proof}

Similarly, there is a lower d\'ecalage $\ldec X$ (or initial path space $P^\triangleleft X$) with $(\ldec X)_n = X_{n+1}$ obtained by forgetting the bottom face and degeneracy operators (and shifting indices by one).
It can be formally obtained using precomposition with the endofunctor $[n] \mapsto [0] \star [n]$ of $\Delta$, sending $f \colon [n] \to [m]$ to $g \colon [n+1] \to [m+1]$ satisfying $g(0) = 0$ and $g(k) = f(k-1) + 1$ for $1 \leq k \leq n+1$.
In other words, the face and degeneracy operators are as follows:
\begin{align*} 
\big( d_k \colon (\ldec X)_n \to (\ldec X)_{n-1} \big) &= \big( d_{k+1} \colon X_{n+1} \to X_n \big)
\\
\big( s_k \colon (\ldec X)_n \to (\ldec X)_{n+1} \big) &= \big( s_{k+1} \colon X_{n+1} \to X_{n+2} \big).
\end{align*}
One checks that $(\udec X)^\op = \ldec(X^\op)$, so that applying \cref{thm upper path} to $X^\op$ we obtain the following.

\begin{corollary}
A simplicial space $X$ is lower 2-Segal if and only if its lower d\'ecalage $\ldec X$ is Segal. \qed
\end{corollary}

Using \cref{thm unitality}, we conclude the original form of the path space criterion.

\begin{corollary}\label{cor path space}
A simplicial space $X$ is a decomposition space if and only if both $\udec X$ and $\ldec X$ are Segal. \qed
\end{corollary}

When $X$ is a decomposition space, it turns out that the maps $d_\top \colon \udec X \to X$ and $d_\bot \colon \ldec X \to X$ are culf (in the sense of \cref{def culf}); see \cite[Theorem~4.10]{GKT1}.

\section{The edgewise subdivision criterion}\label{sec edgewise}

In this section we explain the following theorem from \cite{BOORS:ESC}.
The proof presented here utilizes the exact same techniques that we have been using since \cref{sec 2-Segl decomp}.
This proof relies on the path space criterion (\cref{cor path space}), which is appropriate as there is a close connection between the edgewise subdivision and the upper/lower d\'ecalage pair.\phnote{Additionally, this theorem can be interpreted in terms of the double-categorical perspective from \cite{BOORS:2SSWC}.}
\begin{theorem}[Bergner, Osorno, Ozornova, Rovelli, Scheimbauer]\label{thm edgewise}
A simplicial space $X$ is a decomposition space if and only if its edgewise subdivision $\sd X$ is Segal.
\end{theorem}

As with the d\'ecalage functors, the edgewise subdivision is the result of precomposition with an endofunctor of $\Delta$, namely the one which sends $[n]$ to $[n]^\op \star [n] \cong [2n+1]$.
A map $f \colon [n] \to [m]$ is sent to $g \colon [2n+1] \to [2m+1]$ with
\[
g(t) = \begin{cases}
m - f(n-t) & t \leq n \\
m+1 + f(t-n-1) & n+1 \leq t.
\end{cases}
\]
If we write $[2n+1]$ in the following way
\[
\begin{tikzcd}[sep=small]
  0 \rar & 1 \rar & \cdots \rar & n \\
  0' \uar & 1' \lar & \cdots \lar & n' \lar
\end{tikzcd}
\]
then $g$ can be interpreted as $t' \mapsto f(t)'$ and $t \mapsto f(t)$.
If $X$ is a simplicial space, then the simplicial space $Z = \sd X$ has $Z_n = X_{2n+1}$ and
\begin{align*}
  \big(d_i \colon Z_n \to Z_{n-1}\big) &= \big(d_{n-i} d_{n+i+1} \colon X_{2n+1} \to X_{2n-1}\big) \\
  \big(s_i \colon Z_n \to Z_{n+1}\big) &= \big(s_{n-i} s_{n+i+1} \colon X_{2n+1} \to X_{2n+3}\big).
\end{align*}
One can check that $\sd X = \sd (X^\op)$. 
If $X$ is the nerve of a category $C$, then $\sd X$ is (isomorphic to) the nerve of the \emph{twisted arrow category} of $C$, whose objects are the morphisms of $C$, and where a morphism $f \to g$ is given by a factorization of $g$ as below.
\[
\begin{tikzcd}  
  c' \rar{k} & d' \\
  c \uar{f} & d \lar{h} \uar[swap]{g}.
\end{tikzcd}
\]

One manifestation of the connection between d\'eclage and edgewise subdivision is that iterated bottom and top degeneracies yield simplicial maps
\[ \begin{tikzcd}[sep=small]
\ldec X \rar & \sd X & (\udec X)^\op. \lar
\end{tikzcd} \]
For the map on the left, recall that $\ldec X$ is given by precomposition with the endofunctor $[n] \mapsto [0] \star [n]$. 
The maps $[n]^\op \star [n] \to [0] \star [n]$ collapsing the first component to a point gives a natural transformation of functors, and hence the indicated map.
It is specified by
\begin{equation}\label{eq ldec to sd}
  \ldec X_n = X_{n+1} \xrightarrow{s_\bot^n} X_{2n+1} = \sd X_n.
\end{equation}
Likewise, the leftward arrow is given by
\[
  \sd X_n = X_{2n+1} \xleftarrow{s_\top^n} X_{n+1} = (\udec X)^\op_n,
\]
which arises from the natural transformation $[n]^\op \star [n] \to [n]^\op \star [0]$.

\begin{lemma}\label{lem decomp implies sd Segal}
If $X$ is a decomposition space, then $\sd X$ is Segal.
\end{lemma}
\begin{proof}
Let $Z = \sd X$.
For $n\geq 1$, the following two squares are the same.
\[ \begin{tikzcd}
Z_{n+1} \rar{d_{n+1}} \dar[swap]{d_0}  & Z_n \dar{d_0} 
&  
X_{2n+3} \rar{d_0d_{2n+3}} \dar[swap]{d_{n+1}d_{n+2}}  & X_{2n+1} \dar{d_n d_{n+1}}
\\
Z_n \rar[swap]{d_n} & Z_{n-1}
&  
X_{2n+1} \rar[swap]{d_0d_{2n+1}} & X_{2n-1}
\end{tikzcd} \]
The result follows since the right-hand square is an active-inert pullback square.
\end{proof}

\begin{lemma}\label{lem ldec square as retract}
Let $X$ be a simplicial space and write $Y = \ldec X$, $Z = \sd X$. For each $n \geq 1$, the square below left is a retract of the square below right.
\[ \begin{tikzcd}
Y_{n+1} \rar{d_\bot} \dar[swap]{d_\top} & Y_n \dar{d_\top} 
& Z_{n+1} \rar{d_\bot} \dar[swap]{d_\top} & Z_n \dar{d_\top} 
\\
Y_n \rar[swap]{d_\bot} & Y_{n-1}
&
Z_n \rar[swap]{d_\bot} & Z_{n-1}
\end{tikzcd} \]
Consequently, if $\sd X$ is Segal, so is $\ldec X$.
\end{lemma}
\begin{proof}
The conclusion follows since pullbacks are stable under retracts.
The map from left to right is given by the simplicial map $Y \to Z$ from \eqref{eq ldec to sd}.
We explicitly give the retraction from right to left (which does \emph{not} come from a simplicial map $Z \to Y$).
Utilizing the simplicial identities, all faces in the following cube commute.
\[ \begin{tikzcd}
X_{2n+3} \ar[rr,"d_0d_{2n+3}"] \ar[dd,"d_{n+1}d_{n+2}"'] \ar[dr,"d_0d_2^n" description] & & 
X_{2n+1} \ar[dd, "d_nd_{n+1}"' very near start] \ar[dr,"d_1^n"]  \\
& 
X_{n+2} \ar[rr,"d_{n+2}" near end, crossing over] & & 
X_{n+1} \ar[dd,"d_1"] \\
X_{2n+1} \ar[rr,"d_0d_{2n+1}" very near start] \ar[dr,"d_0d_2^{n-1}"'] & & X_{2n-1} \ar[dr,"d_1^{n-1}" description] \\
& X_{n+1} \ar[rr,"d_{n+1}"] \ar[from=uu, crossing over, "d_1" very near start] & & X_n
\end{tikzcd} \]
In terms of the face maps of $Y$ and $Z$, the front and back faces of the cube are as follows.
\[ \begin{tikzcd}
Z_{n+1} \ar[rr,"d_{n+1}^Z"] \ar[dd,"d_0^Z"'] \ar[dr,"d_0d_2^n" description] & & 
Z_n \ar[dd, "d_0^Z"' very near start] \ar[dr,"d_1^n"]  \\
& 
Y_{n+1} \ar[rr,"d_{n+1}^Y" near end, crossing over] & & 
Y_n \ar[dd,"d_0^Y"] \\
Z_n \ar[rr,"d_n^Z" very near start] \ar[dr,"d_0d_2^{n-1}"'] & & Z_{n-1} \ar[dr,"d_1^{n-1}" description] \\
& Y_n \ar[rr,"d_n^Y"] \ar[from=uu, crossing over, "d_0^Y" very near start] & & Y_{n-1}
\end{tikzcd} \]
We thus see that the $Y$ square is indeed a retract of the $Z$ square since $d_1^n s_0^n = \id_{X_{n+1}} = d_0 d_2^{n-1} s_0^n$ for all $n\geq 1$.
\end{proof}

\begin{lemma}
If $\sd X$ is Segal, so is the upper d\'ecalage $\udec X$.
\end{lemma}
\begin{proof}
We have $\sd X = \sd (X^\op)$ for any simplicial space $X$. 
By \cref{lem ldec square as retract} we conclude that $\ldec(X^\op) = (\udec X)^\op$ is Segal.
\end{proof}

\begin{proof}[Proof of \cref{thm edgewise}]
The forward direction is \cref{lem decomp implies sd Segal}.
If $\sd X$ is Segal, then the preceding two lemmas show that $\ldec X$ and $\udec X$ are Segal as well, so the path space criterion implies that $X$ is a decomposition space.
\end{proof}


\begin{remark}\label{rmk sd vs dec}
If $X$ is a discrete decomposition space, then the three simplicial sets $\sd X$, $\ldec X$, and $\udec X$ are (nerves of) categories.
We have inclusions of simplicial sets $\ldec X \hookrightarrow \sd X$ and $(\udec X)^\op \hookrightarrow \sd X$.
Identifying each of these subcategories with its image, the category $\sd X$ admits two (strict) factorization systems $(\ldec X, (\udec X)^\op)$ and $((\udec X)^\op, \ldec X)$; see \cite{Hackney:OC2SS} for one explanation.
\end{remark}

There is also an interesting story about culf maps (\cref{def culf}) and edgewise subdivision, which gives the following theorem from \cite{KockSpivak:DSST} (see \cite{HackneyKock:CMES} for a generalization). 

\begin{theorem}[Kock \& Spivak]\label{thm kock spivak}
If $X$ is a discrete decomposition space, then there is an equivalence of categories between the category of culf maps over $X$ and the category of presheaves on $\sd X$.
\end{theorem}

\section{Free decomposition spaces}\label{sec free}

We conclude by briefly explaining a class of examples.
For simplicity of presentation, in this section we restrict ourselves to the category of simplicial sets, $\sset$. 
This section is based on \cite{HackneyKock:FDS}, which addresses the more general case of simplicial spaces.

\begin{definition}
An \emph{outer face complex} is a presheaf on the category of inert maps, that is, a functor $A \colon \delint^\op \to \set$.
\end{definition}

As the top and bottom coface maps in $\Delta$ generate the subcategory $\delint$, an outer face complex is given by a sequence of sets and maps as displayed (compare to \eqref{diag simplicial object}):
\begin{equation}
\begin{tikzcd}
A_0 & A_1 \lar[shift left, "d_\bot"] \lar[shift right, "d_\top"']  & A_2 \lar[shift left, "d_\bot"] \lar[shift right, "d_\top"']
 &
A_3 
\lar[shift left, "d_\bot"] \lar[shift right, "d_\top"'] \cdots &[-0.5cm] d_\top d_\bot = d_\bot d_\top.
\end{tikzcd} \end{equation}
Let $j\colon \delint \to \Delta$ be the inclusion of the category of inert maps.
An outer face complex then determines a simplicial set by left Kan extension along $j$, which we denote by $X = j_! A$.
A surprising fact is that this simplicial set is always 2-Segal.

\begin{theorem}
If $A$ is an outer face complex, then $X = j_! A$ is a (discrete) decomposition space.
\end{theorem}

Rather than proving this formally (which is already done in \cite{HackneyKock:FDS}), we instead sketch out the idea about why this might be true.
We begin by unraveling the structure of $X = j_!A$, which we calculate as follows:
\begin{align*}
X_k = j_!A_k &\cong \colim_{([k] \downarrow j)^\op} A_m \\
&\cong\sum_{ [k] \ract [m]} A_m \\
&\cong \sum_{\ell_1, \dots, \ell_k \in \mathbb{N}} A_m & \text{where } m= \ell_1 + \dots + \ell_k.
\end{align*}
The first identification is the standard pointwise colimit formula for the left Kan extension along $j^\op \colon \delint^\op \to \Delta^\op$ (see \cite[X.3]{MacLane:CWM}), along with an isomorphism of categories $(j^\op \downarrow [k]) \cong ([k] \downarrow j)^\op$.
The second identification uses a cofinality argument: the discrete subcategory of $[k] \downarrow j$ on the set of active maps $[k] \ract [m]$ is an initial subcategory. (By the unique factorization in \cref{sec act inert fact}, each active map is an initial object in a connected component of $[k] \downarrow j$.)
Taking opposites makes the inclusion a final functor, so the colimit may instead be computed on this discrete subcategory \cite[IX.3]{MacLane:CWM}.
The last identification simply uses the following bijection:
\begin{align*}
\{ [k] \ract [m] \}  &\cong \{ (\ell_1, \dots, \ell_k) \in \mathbb{N}^k \mid \textstyle  \sum \ell_i = m \} \\
\alpha   &\mapsto (\alpha(1) - \alpha(0), \alpha(2) - \alpha(1), \dots, \alpha(k) - \alpha(k-1)) \\
\textstyle  ( t \mapsto \sum_{i=1}^t \ell_i ) &\mapsfrom (\ell_1, \dots, \ell_k) 
\end{align*}

We now give an indication of the face maps.
Writing an element of $X_k$ as $(a, \ell_1, \dots, \ell_k)$ where $a\in A_m$ (for $m = \sum \ell_i$), the inner face maps are given by adding adjacent natural numbers:
\[
  d_i(a, \ell_1, \dots, \ell_k) = (a, \ell_1, \dots, \ell_i + \ell_{i+1}, \dots, \ell_k) 
\]
for $i=1, \dots, k-1$.
The outer face maps discard an the outer element of the list of natural numbers, and use the structure maps from our original outer face complex $A$:
\begin{align*}
  d_0(a, \ell_1, \dots, \ell_k) &= (d_\bot^{\ell_1} a, \ell_2, \dots, \ell_k) \\
  d_k(a, \ell_1, \dots, \ell_k) &= (d_\top^{\ell_k} a, \ell_1, \dots, \ell_{k-1}).
\end{align*}
Degeneracies are given by inserting a zero into the list of natural numbers.
We leave it to the reader to trace through the identifications above to verify that these formulas are correct.

Formulas for the face maps at hand, it is straightforward to check that $X$ is a 2-Segal set.
For instance the square 
\[ \begin{tikzcd}
X_3 \ar[dr, phantom, "\lrcorner" very near start] \rar{d_2} \dar[swap]{d_0} & X_2 \dar{d_0} \\
X_2 \rar[swap]{d_1} & X_1
\end{tikzcd} \]
is a pullback
since if we have $(a,\ell_1, \ell_2), (b,m_1,m_2) \in X_2$ (with $a\in A_{\ell_1 + \ell_2}$ and $b\in A_{m_1 + m_2}$) such that $d_1(a, \ell_1, \ell_2) = (a, \ell_1 + \ell_2)$ is equal to $d_0(b,m_1, m_2) = (d_\bot^{m_1} b, m_2)$ then the only element $x = (b',k_1, k_2, k_3) \in X_3$ with $d_0 x = (a,\ell_1,\ell_2)$ and $d_2x = (b,m_1,m_2)$ is $x = (b,m_1, \ell_1, \ell_2)$. 
All other pullback squares are verified in a similar manner.

The most basic example is to apply this construction to the terminal outer face complex, where each $A_m = \ast$ is a one-element set.
In this case, \[ X_k = \sum_{\ell_1, \dots, \ell_k \in \mathbb{N}} \ast \,\, \cong \mathbb{N} \times \underset{k}\cdots \times \mathbb{N}.\]
This simplicial set $X$ is isomorphic to $B\mathbb{N}$, the nerve of the monoid of natural numbers.
This example generalizes to other free monoids:

\begin{example}\label{example words}
Fix an alphabet $A_1 \in \set$, and let $A_m$ denote the set of length $m$ words in $A_1$.
The face maps $d_\bot, d_\top \colon A_m \to A_{m-1}$ discard the first or last letters of a word.
Let $X = j_!A$.
The set of vertices $X_0 = A_0$ is a one-point set containing only the empty word.
The set of edges $X_1 = \sum_{m\geq 0} A_m$ is the set of all finite words $(a_1 \dots a_m)$.
We have
\[
  X_2 = \sum_{\ell_1, \ell_2} A_{\ell_1 + \ell_2} = \sum_{\ell_1+\ell_2} A_{\ell_1} \times A_{\ell_2} = \{ (a_1 \dots a_{\ell_1} | a_{\ell_1+1} \dots a_{\ell_1+\ell_2} ) \}
\]
is the set of all words together with a cut point.
The outer face maps forget one side of the cut, while the inner face map forgets the cut point.
This pattern continues, with $X_3$ consisting of words with two cut points, and so on.
This example is Segal -- it is the nerve of the free monoid on $A_1$.
However, we can tweak it by fixing a positive integer $N$ and considering a sub-outer face complex $A'$ consisting only of words of length at most $N$:
\begin{equation*}\label{eq truncation} 
\begin{tikzcd}[row sep=small]
A: &[-0.5cm] A_0 & 
A_1 \lar[shift left, "d_\bot"] \lar[shift right, "d_\top"'] & 
A_2 \lar[shift left, "d_\bot"] \lar[shift right, "d_\top"'] 
\rar[phantom, "\cdots"]  &[-0.5cm] 
A_{N-1} &
A_{N} \lar[shift left, "d_\bot"] \lar[shift right, "d_\top"'] &
A_{N+1} \lar[shift left, "d_\bot"] \lar[shift right, "d_\top"'] \cdots
\\
A': &[-0.5cm] A_0 & 
A_1 \lar[shift left, "d_\bot"] \lar[shift right, "d_\top"'] & 
A_2 \lar[shift left, "d_\bot"] \lar[shift right, "d_\top"'] 
\rar[phantom, "\cdots"]  &[-0.5cm] 
A_{N-1} &
A_{N} \lar[shift left, "d_\bot"] \lar[shift right, "d_\top"'] &
\varnothing \lar[shift left] \lar[shift right] \cdots
\end{tikzcd} \end{equation*}
Then $X' = j_!A'$ is the decomposition space of words of length at most $N$, which is not Segal.
\end{example}

Notice that the decomposition space of words of length at most $N$ also falls under the setting of \cref{ex partial monoid}, by letting $M$ be the subset of the free monoid on $A_1$ with length bounded by $N$.

\begin{example}\label{ex dir graph}
Let $G$ be a directed graph, such as the following.
\[ \begin{tikzcd}[row sep=tiny]
&[-0.2cm] \bullet \rar[bend left,"b"] \rar[bend right,"c"] & \bullet \ar[lld, bend left, "d"]  \\
\bullet \ar[ur, bend left,"a"] \ar[in=200, out=140, loop,"e"']
\end{tikzcd} \]
The associated outer face complex is the one where $A_n$ consists of length $n$ paths in the graph.
Functionally, this is a modification of \cref{example words} where the words are restricted to only allow certain length 2-subwords; for instance $ac$ and $ee$ are acceptable subwords but $bc$ and $eb$ are not.
\end{example}

We leave many other interesting examples to \cite{HackneyKock:FDS}, and conclude with a statement of the main theorem of that paper.

\begin{theorem}
The category of outer face complexes is equivalent to the category of culf maps over $B\mathbb{N}$.
\end{theorem}

The theorem says, roughly, that outer face complexes may be identified with decomposition spaces $X$ equipped with a well-behaved notion of (integer-valued) edge length.
One can classify the CULF maps over $B\mathbb{N}$ whose domains are Segal; these are precisely those categories which are free on a directed graph \cite[\S1]{Street:CS}.
The corresponding outer face complex is exactly the one from \cref{ex dir graph} (see \cite[4.1]{HackneyKock:FDS}).

\bibliographystyle{amsplain}
\bibliography{decomp}
\end{document}